\documentclass[reqno]{proc-l}

\usepackage{enumerate}
\usepackage{xcolor}

\usepackage{mathrsfs}
\usepackage{hyperref}

\newtheorem{theorem}{Theorem}[section]
\newtheorem{lemma}[theorem]{Lemma}

\theoremstyle{definition}
\newtheorem{definition}[theorem]{Definition}
\newtheorem{example}[theorem]{Example}
\newtheorem{corollary}[theorem]{Corollary}

\theoremstyle{remark}
\newtheorem{remark}[theorem]{Remark}

\numberwithin{equation}{section}

\begin{document}

\title[Very slowly growing variance]{Non-local in time telegraph equations and very slowly growing variances}


\author{Francisco Alegr\'ia}
\address{Instituto de Ciencias F\'isicas y Matem\'aticas. Facultad de Ciencias. Universidad Austral de Chile, Valdivia, Chile. Departamento de Matem\'atica y Estad\'istica. Facultad de Ingenier\'ia y Ciencias. Universidad de La Frontera, Temuco, Chile.}
\email{franciscoalegria@uach.cl}

\author{Juan C. Pozo}
\address{Departamento de Matem\'aticas, Facultad de Ciencias, Universidad de Chile, Las Palmeras 3425, \~{N}u\~{n}oa, Santiago, Chile. Partially supported by Fondecyt grant 1181084}
\email{jpozo@uchile.cl}


\date{\today}

\dedicatory{}

\commby{}

\begin{abstract}
In this paper we consider a class of non-local in time telegraph equations. Recently, in \cite{Pozo-Vergara-2019-2} it has been proved that the fundamental solutions of such equations can be interpreted as the probability density function of a stochastic process. We study the asymptotic behavior of the variance of this process at large and short times. In this context, we develop a method to construct new examples such the variance has a slowly growth behavior, extending the earlier results. Finally, we show that our approach can be adapted to define new integro-differential operators which are interesting in sub-diffusion processes. 
\end{abstract}

\maketitle


\section{Introduction}

The study of non-local in time differential equations have received a lot of attention in last years due to its deep connection with non-local transport phenomena, control of stochastic jump processes, description of anomalous diffusion in physics and memory effects in parabolic equations, see \cite{Kemppainen-Siljander-Vergara-Zacher-2016,Kemppainen-Siljander-Zacher-2017,Vergara-Zacher-2015,Pozo-Vergara-2019-2,Pozo-Vergara-2019,Kemppainen-Zacher-2019} and references therein.

Let $k\in L_{1,loc}(\mathbb{R}_+)$ and $\eta,\nu$ be positive constants. In this paper we consider the  non-local in time telegraph equation  
\begin{equation}\label{Eq-Non-Local-k-k}
\partial_t^{2}\bigl(k\ast k\ast u(\cdot,z)\bigr)(t) + \eta\, \partial_t\bigl(k\ast u(\cdot,z)\bigr)(t)-\nu\, \partial_z^2 u(t,z)=0, \ t>0,\ z\in\mathbb{R},
\end{equation}
which has been recently proposed in \cite{Pozo-Vergara-2019-2}. We are interested into study some properties of the fundamental solution of \eqref{Eq-Non-Local-k-k}. For this reason we consider the initial conditions
\begin{equation}\label{Initial-Conditions}
u(0,z)=\delta_0(z)\quad \text{and}\quad \partial_t u(0,z)=0,\ z\in\mathbb{R},
\end{equation}
where the $\partial_t u(0,z)=0$ must be considered whenever it exists. We point out that the fundamental solution of \eqref{Eq-Non-Local-k-k} has been already studied in \cite[Section 4]{Pozo-Vergara-2019-2} assuming that $k$ is a kernel of type $(\mathcal{PC})$, which means that the following condition is satisfied.
\begin{enumerate}[$(\mathcal{PC})$]
\item $k\in L_{1,loc}(\mathbb{R}_+)$ is nonnegative, nonincreasing, and there is $\ell\in L_{1,loc}(\mathbb{R}_+)$ such that $(k*\ell) = 1$ in $(0, \infty)$. In this case we also write $(k,\ell)\in(\mathcal{PC})$.
\end{enumerate} 
In this paper we also assume that $k$ is a kernel of type $(\mathcal{PC})$ and we extend some of the results obtained in \cite{Pozo-Vergara-2019-2}.

It is worth mentioning that the most classical example of a pair $(k,\ell)\in(\mathcal{PC})$ is given by $(g_{1-\alpha},g_\alpha)$ with $\alpha\in(0,1)$, where $g_\beta$ with $\beta>0$ is the standard notation for the function
\[
g_{\beta}(t)=\dfrac{t^{\beta-1}}{\Gamma(\beta)},\quad t>0.
\] 
In this case \eqref{Eq-Non-Local-k-k}-\eqref{Initial-Conditions} takes the form of the {\it time fractional telegraph equation} 
\begin{align}
\partial_t^{2\alpha}u(t,z) + \eta\, \partial_t^\alpha u(t,z)-\nu\, \partial_z^2 u(t,z)&=0,\ t>0,\ z\in\mathbb{R},\label{Eq-Non-Local-fractional}\\
u(0,z)=\delta_0(z)\quad \text{and}\quad \partial_t u(0,z)&=0,\ z\in\mathbb{R},\label{Frac-Init-Cond}
\end{align}
which have been extensively studied, for instance see \cite{Orsingher-Beghin-2004,Vergara-2014} and references therein. In this case the initial condition on the derivative of $u$ only must be considered if $\alpha\in(\frac{1}{2},1]$. The solution $U_\alpha$ of \eqref{Eq-Non-Local-fractional}-\eqref{Frac-Init-Cond} exhibits several interesting properties, one of them being that $U_\alpha$ can be viewed as the probability density function of a stochastic process denoted by $X_\alpha(t)$, (cf. \cite{Orsingher-Beghin-2004}). Moreover, the variance of the process $X_\alpha(t)$ is given by 
\[
\textrm{Var}[X_\alpha(t)]=2\nu t^{2\alpha}E_{\alpha,2\alpha+1}(-\eta t^{\alpha}),\ t\ge 0,
\]
where $E_{\alpha,2\alpha+1}$ is an example of the so-called Mittag-Leffler function of two parameters, see \cite[Chapter 4]{Gorenflo-Kilbas-Mainardi-Rogosin-2014} or \cite[Appendix E]{Mainardi-2010} for several properties and results of this function.

Motivated by this result, it has been established in \cite[Theorem 1.1]{Pozo-Vergara-2019-2} that for every $(k,\ell)\in(\mathcal{PC})$ the fundamental solution $U(t,z)$ of \eqref{Eq-Non-Local-k-k} can be interpreted as a probability density function on $\mathbb{R}$ and there exists a process, denoted by $X(t)$, whose distribution
coincides with $U(t,\cdot)$ for all time $t>0$. Further, the corresponding variance of this process is positive, increasing on $(0,\infty)$ and it is given by
\begin{equation}\label{Var-0}
\mathrm{Var}[X(t)]=2\nu(1\ast r_{\eta}\ast\ell)(t), \quad t\ge 0,
\end{equation}
where $r_\eta$ is the {\it integrated resolvent associated to $\ell$}, (see Definition \ref{Scalar:Resolvent} below). 

The importance of knowing the variance of a stochastic process relies in the fact that allows measuring how far the set of random values are spread out from their average. Although \eqref{Var-0} provides an exact  representation of $\textrm{Var}[X(t)]$, in general the function $r_\eta$ cannot be computed explicitly. In consequence, get more analytical properties of the variance could be a hard task. For example, we are particularly interested in knowing how slow its growth rate can be. In this context, in \cite[Section 4]{Pozo-Vergara-2019-2} it has been proved that the asymptotic behavior of $\textrm{Var}[X(t)]$ could be of very different kinds, which are e.g., exponential, algebraic and logarithmic. To the best of our knowledge, the slowest growth rate of $\textrm{Var}[X(t)]$ known in the literature is logarithmic, and this rate is obtained considering 
\begin{equation}\label{k-l-log}
k(t)=\int_{0}^1 \dfrac{t^{\alpha-1}}{\Gamma(\alpha)}d\alpha,\quad \text{and}\quad\ell(t)=\int_0^\infty \dfrac{e^{-st}}{1+s}ds,\quad t>0.
\end{equation}

So a basic question arises: {\it Is there a pair $(k,\ell)\in(\mathcal{PC})$ such that the variance grows slower than a logarithmic function at infinity?} In this paper we show that the answer to this question is affirmative. Indeed, we develop a method to construct infinitely many examples of pairs $(k,\ell)\in(\mathcal{PC})$  answering affirmatively this question. These examples cover a broad range of slowly growing functions. 

The paper is organized as follows. In Section 2, we give some preliminaries concept that we need to our work. Section 3 is the central part of this work. In this section we prove the main results of our work. In Section 4, we explain why this method could be interesting in another contexts such as sub-diffusion processes.

\section{Preliminaries}
The Laplace transform of a subexponential function $f\colon [0,\infty)\to \mathbb{R}$ defined on the half-line will be denoted by
\[
\widehat{f}(\lambda)=\int_0^{\infty} e^{-\lambda t}f(t)dt,\quad \lambda>0.
\] 


\begin{definition}\label{Def:Mono}
An infinitely differentiable function $f\colon (0,\infty)\to\mathbb{R}$ is called {\it completely monotone} if $(-1)^n f^{(n)}(\lambda)\ge 0$ for all $n\in\mathbb{N}_0$ and $\lambda > 0$. An infinitely differentiable function $f\colon (0,\infty)\to\mathbb{R}$ is called {\it Bernstein function} if $f(\lambda)\ge0$ for $\lambda>0$ and $f'$ is a completely monotonic function. We will denote the class of completely monotonic functions by $(\mathcal{CM})$, and the class of Bernstein functions will be denoted by $(\mathcal{BF})$. 
\end{definition}

\noindent A detailed collection of the most important properties and results about the classes $(\mathcal{CM})$ and $(\mathcal{BF})$ can be found in  \cite[Chapter 3]{Jacob-1996} and \cite{Schilling-Song-Vondracek-2010}.

\begin{definition}\label{Scalar:Resolvent} Let $\eta\in\mathbb{C}$ and $\ell\in L_{1,loc}(\mathbb{R}_+)$. The solution $r_\eta \colon \mathbb{R}_+\to \mathbb{C}$ of the scalar Volterra equation 
\begin{equation}\label{Eq:r}
r_\eta(t) +\eta(r_\eta\ast \ell)(t)=\ell(t),\quad t>0,
\end{equation}
is called {\it integral scalar resolvent associated to $\ell$}.
\end{definition}

\noindent The integral scalar resolvent $r_\eta $ has proved to be essential for the treatment of non-homogeneous Volterra equations, see \cite{Clement-Nohel-1979,Clement-Nohel-1981,Gripenberg-Londen-Staffans-1990,Pozo-Vergara-2019} and references therein. 

\smallskip

We remark that if $\eta\in\mathbb{R}_+$ then the condition $(k,\ell)\in(\mathcal{PC})$ implies that $r_\eta$ is positive, see \cite[Proposition 4.5]{Pruss-1993}. Further, if additionally $\ell$ is completely monotonic then $r_\eta$ is completely monotonic as well, (see \cite[Lemma 4.1]{Pruss-1993}). More properties and results about $r_\eta $ can be found in \cite[Chapter 5]{Gripenberg-Londen-Staffans-1990}. 

Now we introduce a version of Karamata-Feller Tauberian theorem. The proof  can be found in \cite[Section 1.7, Chapter I]{Bingham-Goldie-Teugels-1989} or \cite[Chapter XIII]{Feller-1971}. 


\begin{definition}\label{Regularly-Varying-Functions}
Let $\varrho\in\mathbb{R}$. We say that a function $L:(0,\infty)\to (0,\infty)$ is a {\it regularly varying function of index $\varrho$} if for every fixed $x>0$ we have that 
\[
\lim_{t\to\infty }\dfrac{L(tx)}{L(t)}=x^\varrho.
\]
The regularly varying functions at infinity of index $\varrho=0$ are called {\it slowly varying functions}.  
\end{definition}

\begin{remark}\label{Remark:Regularly} Let $F\colon(0,\infty)\to(0,\infty)$ be a regularly varying function at infinity of index $\varrho$. It follows from \cite[Theorem 1.4.1]{Bingham-Goldie-Teugels-1989} that there is a slowly varying function $L\colon(0,\infty)\to(0,\infty)$ such that $F(x)=x^{\varrho}L(x)$ for $x>0$.
\end{remark}


\begin{theorem}[Karamata-Feller Theorem] \label{Karamata:Feller}
Let $L_1,L_2\colon(0,\infty)\to (0,\infty)$ slowly varying functions. Let $\beta>0$ and $w:(0,\infty)\rightarrow \mathbb{R}$ be a monotone function whose Laplace transform $\widehat{w}(\lambda)$ exists for all $\lambda\in \mathbb{C}_+$. Then
\[
\widehat{w}(\lambda) \sim \,\frac{1}{\lambda^\beta}\,L_1(\lambda),\ \text{as}\;\lambda\to \infty,\ \; \mbox{if and only if} \ \;
w(t)\sim \frac{t^{\beta-1}}{\Gamma(\beta)}L_1\left(\dfrac{1}{t}\right),\ \mbox{as}\;t\to 0^+,
\]
and 
\[
\widehat{w}(\lambda) \sim \,\frac{1}{\lambda^\beta}\,L_2\left(\frac{1}{\lambda}\right),\ \mbox{as}\ \lambda\to 0^+,\ \; \mbox{if and only if} \ \
w(t)\sim \frac{t^{\beta-1}}{\Gamma(\beta)}L_2(t),\ \mbox{as}\ t\to \infty.
\]
Here the approaches are on the positive real axis and the notation $f(t)\sim g(t)$ as $t\to t_*$ means that $\lim_{t\to t_*} f(t)/g(t) =1$. 
\end{theorem}

\section{Very slowly growing variances}

We begin this section recalling the result established in \cite[Theorem 1.1]{Pozo-Vergara-2019-2} which is the base of our research.
\begin{theorem}\label{Theo-PDF}Let $\eta,\nu$ positive constants and $(k,\ell)\in(\mathcal{PC})$. The variance of the process $X(t)$, whose density function coincides with the fundamental solution of \eqref{Eq-Non-Local-k-k}, satisfies the following Volterra equation
\begin{equation}
\textrm{Var}[X(t)]+\eta(\ell\ast \textrm{Var}[X(\cdot)])(t)=2\nu(1\ast\ell\ast\ell)(t),\quad t\ge 0.
\end{equation}
Further, $\textrm{Var}[X(t)]$  is positive and increasing on $(0,\infty)$ and it satisfies the formula
\begin{equation}\label{Representation:Var}
\textrm{Var}[X(t)]=2\nu(1\ast\ell\ast r_\eta)(t),\quad t\ge 0.
\end{equation}
\end{theorem}

\noindent Since \eqref{Representation:Var} is given by means of convolutions, we can use Theorem \ref{Karamata:Feller} to study the behavior of $\mathrm{Var}[X(t)]$ at large times and short times. The following theorem is the main results of our work.


\begin{theorem}\label{Theo-Asymptotic-Behavior}Let $(k,\ell)\in(\mathcal{PC})$. If the Laplace transform $\widehat{\ell}$ is a regularly varying function of index $\varrho_1<\frac{1}{2}$, then 
\begin{equation}\label{Var->0}
\mathrm{Var}[X(t)]\sim \dfrac{2\nu}{\Gamma(1-2\varrho_1)}\left( \widehat{\ell}\bigl(t^{-1}\bigr)\right)^2,\ \text{as}\ t\to0^+.
\end{equation}
Further, if the function $t\mapsto\widehat{\ell}(t^{-1})$ is a regularly varying function of index $\varrho_2>-1$, then 
\begin{equation}\label{Var->00}
\mathrm{Var}[X(t)]\sim \dfrac{2\nu}{\eta\,\Gamma(1+\varrho_2)}\,\widehat{\ell}\Bigl(\frac{1}{t}\Bigr),\quad \text{as}\quad t\to \infty.
\end{equation}
\end{theorem}

\begin{proof}
Set $V(t)=\mathrm{Var}[X(t)]$ for $t\ge 0$. It follows from \eqref{Representation:Var} that the Laplace transform of $V$ is given by  
\[
\widehat{V}(\lambda)=\dfrac{2\nu}{\lambda}\cdot \dfrac{\widehat{\ell}(\lambda)}{1+\eta\widehat{\ell}(\lambda)}\cdot \widehat{\ell}(\lambda),\quad \lambda>0.
\]
Since $\widehat{\ell}(\lambda)\to0$ as $\lambda\to \infty$, we have that 
\[
\widehat{V}(\lambda)\sim\dfrac{2\nu}{\lambda^{1-2\varrho_1}} L_1(\lambda),\quad \text{as}\quad\lambda\to\infty,
\]
where $L_1(t)=t^{-2\varrho_1}\bigl(\widehat{\ell}(t)\bigr)^2$. Since $\widehat{\ell}$ is a regular variation function of index $\varrho_1$, by Remark \ref{Remark:Regularly}, we have that $L_1$ is a slowly varying function. Since $\frac{1}{2}>\varrho_1$, it follows from Theorem \ref{Karamata:Feller} that 
\[
\mathrm{Var}[X(t)]\sim \dfrac{2\nu}{\Gamma(1-2\varrho_1)}\left(\widehat{\ell}\bigl(t^{-1}\bigr)\right)^2,\ \text{as}\ t\to0^+.
\]
On the other hand, we note that
\[
\widehat{V}(\lambda)=\dfrac{2\nu}{\lambda^{1+\varrho_2}}L_2\left(\dfrac{1}{\lambda}\right),\quad \lambda>0.
\]
where $L_2\colon(0,\infty)\to(0,\infty)$ is defined by 
\[
L_2(t)=\frac{\widehat{\ell}(t^{-1})}{1+\eta\widehat{\ell}(t^{-1})}\cdot \frac{\widehat{\ell}(t^{-1})}{t^{\varrho_2}},\quad\text{for}\quad t>0.
\] 
Since $t\mapsto\widehat{\ell}(t^{-1})$ is a regularly varying function of index $\varrho_2$, by Remark \ref{Remark:Regularly} we have that $L_2$ is a slowly varying function. Furthermore, 
\[
L_2(\lambda^{-1})\sim \frac{1}{\eta\lambda^{\varrho_2}}\widehat{\ell}(\lambda^{-1}) \quad\text{as} \quad\lambda\to0^+.
\] 
Since $\varrho_2>-1$, it follows from Theorem \ref{Karamata:Feller} that 
\[
\mathrm{Var}[X(t)]\sim \dfrac{2\nu}{\eta}\cdot\dfrac{t^{\varrho_2}}{\Gamma(1+\varrho_2)} L_2(t)\sim \dfrac{2\nu}{\eta\,\Gamma(1+\varrho_2)}\,\widehat{\ell}\Bigl(\frac{1}{t}\Bigr),\quad \text{as}\quad t\to \infty. 
\]
\end{proof}
\noindent We remark that the conditions of Theorem \ref{Theo-Asymptotic-Behavior} are satisfied for many pairs of functions $(k,\ell)\in(\mathcal{PC})$. For instance, all the examples presented in \cite[Section 5]{Pozo-Vergara-2019-2} satisfy these conditions. For the sake of brevity of the text we omit their proof. 


\begin{example}\label{Example-1}{\it (Time fractional telegraph equation)}. If $k=g_{1-\alpha}$ with $\alpha\in(0,1)$, then 
\[
\textrm{Var}[X(t)]\sim \dfrac{2\nu}{\eta}\dfrac{t^\alpha}{\Gamma(1+\alpha)}, \quad \text{as}\quad t\to \infty,
\]
and 
\[
\textrm{Var}[X(t)]\sim 2\nu\,\dfrac{t^{2\alpha}}{\Gamma(1+2\alpha)}, \quad \text{as}\quad t\to 0^+.
\]
\end{example}


\begin{example}\label{Example-2}{\it (Sum of two time fractional derivatives)}. If $k=g_{1-\alpha}+g_{1-\beta}$ with $0<\alpha<\beta<1$, then 
\[
\textrm{Var}[X(t)]\sim \dfrac{2\nu}{\eta}\dfrac{t^\alpha}{\Gamma(1+\alpha)}, \quad \text{as}\quad t\to \infty,
\]
and 
\[
\textrm{Var}[X(t)]\sim 2\nu\,\dfrac{t^{2\beta}}{\Gamma(1+2\beta)}, \quad \text{as}\quad t\to 0^+.
\]
\end{example}


\begin{example}\label{Example-3}{\it (Time-fractional telegraph equation with Mittag-Leffler weight)}. If $k(t)=t^{\beta-1}E_{\alpha,\beta}(-\omega t^{\alpha})$ with $0<\alpha,\beta<1$ and $\omega>0$, then 
\[
\textrm{Var}[X(t)]\sim \dfrac{2\nu}{\eta}\dfrac{\omega\, t^{\alpha+1-\beta}}{\Gamma(2+\alpha-\beta)}, \quad \text{as}\quad t\to \infty,
\]
and 
\[
\textrm{Var}[X(t)]\sim 2\nu\,\dfrac{t^{2-2\beta}}{\Gamma(2+\alpha-\beta)}, \quad \text{as}\quad t\to 0^+.
\]

\end{example}

\begin{example}\label{Example-4}{\it (Time distributed order telegraph equation)}. If $\displaystyle k(t)=\int_a^b g_{\alpha}(t)d\alpha$, with $0\le a<b\le 1$, then  
\[
\textrm{Var}[X(t)]\sim \dfrac{2\nu}{\eta}\dfrac{t^{1-b}\log(t)}{\Gamma(2-b)}, \quad \text{as}\quad t\to \infty,
\]
and 
\[
\textrm{Var}[X(t)]\sim 2\nu\,\dfrac{t^{1+b-2a}\bigl(\log(t)\bigr)^2}{\Gamma(2-b)}, \quad \text{as}\quad t\to 0^+.
\]
\end{example}

\begin{remark} We note that if $b=1$ and $a\in(0,1)$ in Example \ref{Example-4} then we have an  infinity family of processes whose variance behaves like a logarithmic function at infinity. 
\end{remark}


The following example has not been established in the literature before. In order to present it, we define recursively the functions $\Theta_{n}$ as follows 
\begin{equation}\label{Theta:n}
\Theta_{1}(t,x)=\int_0^x g_{\alpha}(t)d\alpha\quad \text{and}\quad\Theta_{n+1}(t,x)=\int_{0}^{x} \Theta_{n}(t,y)dy,\quad t>0, \ x>0.
\end{equation}
Further, for $n\in\mathbb{N}$ we define the functions $\theta_n\colon(0,\infty)\to (0,\infty)$ by
\begin{equation}\label{theta:n}
\theta_n(t)=\Theta_{n}(t,1),\ \text{for}\ t>0.
\end{equation}

\begin{example}\label{Example-5}Let $n\in\mathbb{N}$. Consider $\displaystyle k=\theta_n$ where $\theta_n$ has been defined in \eqref{theta:n}, then  
\[
\textrm{Var}[X(t)]\sim \dfrac{2\nu}{\eta}\bigl(\log(t)\bigr)^n, \quad \text{as}\quad t\to \infty,
\]
and 
\[
\textrm{Var}[X(t)]\sim \nu \big((n-1)!\bigr)^2\bigl(t\cdot \log(t)\bigr)^2, \quad \text{as}\quad t\to 0^+.
\]

\end{example}

\begin{proof}Let $n\in\mathbb{N}$. We note that $\theta_n$ is positive and locally integrable on $\mathbb{R}_+$. Since $\alpha\in(0,1)$ we have that $g_\alpha\in(\mathcal{CM})$. We recall that the class of completely monotonic functions is closed under addition and pointwise limits, see \cite[Corollary 1.6 and Corollary 1.7]{Schilling-Song-Vondracek-2010}. Therefore, we have that $\theta_{n}\in(\mathcal{CM})$ as well. It follows from \cite[Theorem 5.4 and Theorem 5.5]{Gripenberg-Londen-Staffans-1990} that there exists $\zeta_{n}\in(\mathcal{CM})$ such that 
\begin{equation}\label{zeta-n}
\theta_n\ast\zeta_{n}=1. 
\end{equation}
Consequently, we have that $(\theta_n,\zeta_n)\in(\mathcal{PC})$. Further, for $\lambda>0$ and $x>0$ we note that $\widehat{\Theta}_1(\lambda,x)=\dfrac{\lambda^x-1}{\lambda^x\log(\lambda)}$. Hence, it follows from \eqref{Theta:n} and an inductive procedure that 
\begin{equation*}
\widehat{\Theta}_{n}(\lambda,x)=\dfrac{1}{\lambda^x\bigl(\log(\lambda)\bigr)^n}\left((-1)^n+\lambda^x\sum_{k=1}^n \dfrac{(-1)^{k-1}x^{n-k}(\log(\lambda))^{n-k}}{(n-k)!}\right), \ \lambda>0, \ x>0,
\end{equation*}
for $n\ge 2$. This in turn implies that 
\[
\widehat{\theta}_n(\lambda)=\dfrac{1}{\lambda\bigl(\log(\lambda)\bigr)^n}\left((-1)^n+\lambda\sum_{k=1}^n \dfrac{(-1)^{k-1}(\log(\lambda))^{n-k}}{(n-k)!}\right), \quad \lambda>0,
\]
and
\[
\widehat{\zeta}_n(\lambda)=\dfrac{\bigl(\log(\lambda)\bigr)^n}{\displaystyle (-1)^n+\lambda\sum_{k=1}^n \dfrac{(-1)^{k-1}(\log(\lambda))^{n-k}}{(n-k)!}},\quad \lambda>0,\ 
\]
for $n\ge 2$. We note that $\widehat{\zeta}_n(\lambda)$ can be rewritten as follows 
\[
\widehat{\zeta}_n(\lambda)=\dfrac{\log(\lambda)}{\displaystyle \frac{(-1)^n}{\big(\log(\lambda)\big)^{n-1}}+\frac{\lambda}{(n-1)!}+\lambda\sum_{k=2}^n \dfrac{(-1)^{k-1}}{(n-k)!(\log(\lambda))^{k-1}}},\quad \lambda>0
\]
Hence, $\widehat{\zeta}_n(\lambda)\sim \frac{(n-1)!\log(\lambda)}{\lambda}$ as $\lambda\to \infty$. This implies that $\widehat{\zeta}_n$ is a regularly varying function of index $\varrho=-1$. On the other hand, we have that
\[
\widehat{\zeta}_n(t^{-1})=\dfrac{(-1)^{n}\big(\log(t)\big)^n}{\displaystyle (-1)^n+\frac{(-1)^{n-1}}{t}\sum_{k=1}^n \frac{\big(\log(t)\big)^{n-k}}{(n-k)!}},\quad t>0,
\]
which implies that $\widehat{\zeta}_n(t^{-1})\sim \big(\log(t)\big)^n$ as $t\to\infty$. Since the logarithmic function is regularly varying, and the multiplication of regularly varying functions is regularly varying as well, this in turn implies that the function $t\mapsto \zeta_n(t^{-1})$ is a regularly varying function. Moreover, 
\[
\widehat{\zeta}_n(t^{-1})=\dfrac{(-1)^{n}\,t\log(t)}{\displaystyle \frac{(-1)^n\, t}{(\log(t))^{n-1}}+\frac{(-1)^{n-1}}{(n-1)!}+\sum_{k=2}^n \frac{1}{(n-k)!\,\big(\log(t)\big)^{k-1}}},\quad t>0,
\]
which implies that $\widehat{\zeta}_n(t^{-1})\sim (-1)(n-1)!\, t\log(t)$ as $t\to 0^+$. Consequently, as direct application of Theorem \ref{Theo-Asymptotic-Behavior} we obtain that 
\[
\textrm{Var}[X(t)]\sim \dfrac{2\nu}{\eta}\bigl(\log(t)\bigr)^n, \quad \text{as}\quad t\to \infty,
\]
and 
\[
\textrm{Var}[X(t)]\sim \nu \big((n-1)!\bigr)^2(\bigl(t\cdot \log(t)\bigr)^2, \quad \text{as}\quad t\to 0^+.
\]
\end{proof}


\noindent We remark that all the examples above are very interesting, However, as we have mentioned in the Introduction, we are interested into find examples of $(k,\ell)\in(\mathcal{PC})$ such that $\textrm{Var}[X(t)]$ grows slower than a logarithmic function at infinity. To this end we prove the following result.

\begin{lemma}\label{Lemma-Comp-f-g} Let $f,g\in L_{1,loc}(\mathbb{R}_+)$. Assume that $f,g\in(\mathcal{CM})$, then there exists $h\in(\mathcal{CM})$ such that 
\[
\widehat{h}(\lambda)=\widehat{f}(\lambda)\ \widehat{g}\bigl(\widehat{f}(\lambda)\bigr), \quad \lambda>0.
\]
\end{lemma}
\begin{proof}Consider $a=1\ast f$, $b\equiv 1$, $c=1\ast g$. It is clear that $a,b$ and $c$ are Bernstein functions. According to \cite[Lemma 4.3]{Pruss-1993}, there exists a Bernstein function $e\colon(0,\infty)\to (0,\infty)$ such that $e(0^+)=a(0^+)=0$ and
\[
\widehat{e}(\lambda)=\widehat{a}(\lambda)\widehat{dc}\left(\dfrac{\widehat{a}(\lambda)}{\widehat{b}(\lambda)}\right),\ \lambda>0,
\]
where $\widehat{dc}$ stands for the Laplace transform of the measure $dc$ which is defined by
\[
\widehat{dc}(\lambda)=\int_0^\infty e^{-\lambda t}dc(t),\quad \lambda>0.
\]
Since $c=1\ast g$, we have that $dc(t)=g(t)dt$, and consequently $\widehat{dc}=\widehat{g}$. Let us now define $h(t)=\frac{d}{dt}e(t)$ for $t>0$. Since $e\in(\mathcal{BF})$ we have that $h\in(\mathcal{CM})$ and  
\[
\widehat{h}(\lambda)=\lambda \widehat{e}(\lambda) -e(0^+)=\widehat{f}(\lambda)\ \widehat{g}\bigl(\widehat{f}(\lambda)\bigr), \quad \lambda>0.
\]
\end{proof}


\begin{corollary}\label{Coro-phi-delta}For all $\delta\in(0,1]$ there exists a pair $(\phi^\delta_1,\psi^\delta_1)\in(\mathcal{PC})$ such that 
\[
\widehat{\psi^\delta_1}\bigl(t^{-1}\bigr)\sim \bigl(\log(t)\bigr)^\delta,\quad \text{as}\quad t\to \infty,
\] 
and
\[
\widehat{\psi^\delta_1}\bigl(t^{-1}\bigr)\sim (t\cdot \log(t^{-1})\bigr)^{\delta}\quad \text{as}\quad t\to 0^+.
\]
If $\delta=1$, we will simply write $(\phi_1,\psi_1)\in(\mathcal{PC})$.
\end{corollary}

\begin{proof}
Let $\delta\in(0,1]$. Consider the pair $(\theta,\zeta)\in (\mathcal{PC})$ given by 
\begin{equation}\label{theta-zeta-1}
\theta(t)=\int_{0}^1 \dfrac{t^{\alpha-1}}{\Gamma(\alpha)}d\alpha,\quad \text{and}\quad\zeta(t)=\int_0^\infty \dfrac{e^{-st}}{1+s}ds,\quad t>0.
\end{equation}
It is a well known fact that both $\theta$ and $\zeta$ are completely monotonic functions. So, applying Lemma \ref{Lemma-Comp-f-g} with $f=\theta$ and $g=g_{1-\delta}$, we conclude that there exists $h_{1\delta}\in(\mathcal{CM})$ such that 
\[
\widehat{h_{1\delta}}(\lambda)=\left(\dfrac{\lambda-1}{\lambda\,\log(\lambda)}\right)^{\delta},\quad \lambda>0.
\]
Applying again Lemma \ref{Lemma-Comp-f-g} with $f=\zeta$ and $g=g_{1-\delta}$, we note that there exists $h_{2\delta}\in(\mathcal{CM})$ such that 
\[
\widehat{h_{2\delta}}(\lambda)=\left(\dfrac{\log(\lambda)}{\lambda-1}\right)^\delta, \quad \lambda>0.
\]
Now define the kernels
\begin{equation*}
\phi^\delta_1=g_{1-\delta}\ast h_{2\delta},\quad \text{and}\quad  \psi^\delta_1=h_{1\delta}.
\end{equation*}
Applying directly the Laplace transform, we have 
\begin{equation}\label{theta:delta:log}
\widehat{\phi^\delta_1}(\lambda)=\dfrac{1}{\lambda}\left(\dfrac{\lambda-1}{\log(\lambda)}\right)^\delta,\quad\text{and }\quad \widehat{\psi^\delta_1}(\lambda)=\left(\dfrac{\log(\lambda)}{\lambda-1}\right)^\delta,\quad \lambda>0.
\end{equation}

We note that by construction $\psi^\delta_1\in(\mathcal{CM})$. Hence, it follows from \cite[Theorem 5.4 and Theorem 5.5]{Gripenberg-Londen-Staffans-1990} that $\phi^\delta_1\in(\mathcal{CM})$. In consequence $(\phi^\delta_1,\psi^\delta_1)\in(\mathcal{PC})$. Furthermore, it clear that 
\[
\widehat{\psi^\delta_1}(t^{-1})=\left(\dfrac{\log(t)}{1-t^{-1}}\right)^\delta,\quad t>0,
\]
which in turn implies that 
\[
\widehat{\psi^\delta_1}(t^{-1})\sim\bigl(\log(t)\bigr)^\delta,\quad \text{as}\quad t\to \infty.
\]
To compute the behavior $\widehat{\psi^\delta_1}(t^{-1})$ as $t\to0^+$, we rewrite this function as follows
\[
\widehat{\psi}_\delta\bigl(t^{-1}\bigr)=\left(\dfrac{t\,\log(t)}{t-1}\right)^\delta,\quad t>0,
\]
which implies that
\[
\widehat{\psi}_\delta\bigl(t^{-1}\bigr)\sim (-t\cdot \log(t)\bigr)^{\delta}\quad \text{as}\quad t\to 0^+.
\]

\end{proof}

\begin{remark} Let $\delta\in(0,1]$ and $(\phi_\delta,\psi_\delta)\in(\mathcal{PC})$ given in Corollary \ref{Coro-phi-delta}. If $\delta=1$, then $(\phi_1,\psi_1)$ is the same pair of functions defined by Kochubei in \cite{Kochubei-2008}. On the other hand, since $\widehat{\psi^\delta_1}(t^{-1})\sim (\log(t))^{\delta}$ as $t\to \infty$ we have that $\psi^\delta_1\notin L_1(\mathbb{R}_+)$.
\end{remark}


\begin{example}\label{Example-6}Let $\delta\in(0,1]$. Consider the pair $\displaystyle (k,\ell)=(\phi^\delta_1,\psi^\delta_1)$ given in Corollary \ref{Coro-phi-delta}, then  
\[
\textrm{Var}[X(t)]\sim \dfrac{2\nu}{\eta}\bigl(\log(t)\bigr)^\delta, \quad \text{as}\quad t\to \infty,
\]
and 
\[
\textrm{Var}[X(t)]\sim 2\nu \dfrac{\bigl(t\cdot \log(t^{-1})\bigr)^{2\delta}}{\Gamma(1+2\delta)}, \quad \text{as}\quad t\to 0^+.
\]
\begin{proof} It follows from \eqref{theta:delta:log} that $\widehat{\ell}$ is a regularly varying function of index $\varrho=-\delta$. Further, since 
\[
\widehat{\ell}(t^{-1})\sim (\log(t))^\delta,\ \text{as} \ t\to \infty,
\]
it follows that $\widehat{\ell}(t^{-1})$ is a slowly varying function. Therefore,  Theorem \ref{Theo-Asymptotic-Behavior} implies that 
\[
\textrm{Var}[X(t)]\sim \dfrac{2\nu}{\eta}\bigl(\log(t)\bigr)^\delta, \quad \text{as}\quad t\to \infty,
\]
and
\[
\textrm{Var}[X(t)]\sim 2\nu \dfrac{\bigl(t\cdot \log(t^{-1})\bigr)^{2\delta}}{\Gamma(1+2\delta)}, \quad \text{as}\quad t\to 0^+.
\]

\end{proof}

\end{example}



\begin{corollary}\label{Coro-phi-psi-n}For all $n\in\mathbb{N}$ there exists a pair $(\phi_{n},\psi_{n})\in(\mathcal{PC})$ such that $\phi_{n}\in(\mathcal{CM})$ and 
\[
\widehat{\psi_{n}}(t^{-1})\sim\log^{[n]}(t),\ \text{as}\ t\to\infty,
\]
and 
\[
\widehat{\psi_{n}}(t^{-1})\sim t\, \bigl(\log(t^{-1})\bigr)^{n}\ \text{as}\ t\to0^+,
\]
where $\log^{[n]}=\underbrace{\log\circ\log\circ\dots\circ\log}_{n-\rm{times}}$.
\end{corollary}

\begin{proof}The proof  will be done by induction. For $n=1$, we consider the pair $(\phi_{1},\psi_{1})\in(\mathcal{PC})$ given by Corollary \ref{Coro-phi-delta} with $\delta=1$. 

\noindent Let $n>1$. Assume that there exists a pair $(\phi_{n},\psi_{n})\in(\mathcal{PC})$ such that $\phi_{n}\in(\mathcal{CM})$ satisfying 
\[
\widehat{\psi}_{n}(t^{-1})\sim\log^{[n]}(t),\ \text{as}\ t\to\infty.
\]
and 
\[
\widehat{\psi_{n}}(t^{-1})\sim t\, \bigl(\log(t^{-1})\bigr)^{n},\ \text{as}\ t\to0^+.
\]
Since $\phi_{n}\in(\mathcal{CM})$ and $\phi_{n}\ast \psi_{n}=1$, it follows from \cite[Theorem 5.4 and Theorem 5.5]{Gripenberg-Londen-Staffans-1990} that $\psi_n\in(\mathcal{CM})$. Hence, applying Lemma \ref{Lemma-Comp-f-g} with $f=\psi_n$ and $g=g_{1-\delta}$ with $\delta\in(0,1)$, we conclude that for all $\delta\in(0,1)$ there exists $\omega_n^\delta\in(\mathcal{CM})$ such that 
\[
\widehat{\omega_n^\delta}(\lambda)=\big(\widehat{\psi}_{n}(\lambda)\big)^{\delta},\quad \lambda>0.
\]
Using again \cite[Theorem 5.4 and Theorem 5.5]{Gripenberg-Londen-Staffans-1990} we can establish the existence of $\varphi^\delta_{n}\in(\mathcal{CM})$ such that $\omega_n^\delta\ast \varphi_n^\delta=1$, which in turn implies that
\[
\widehat{\varphi_{n}^\delta}(\lambda)=\frac{1}{\lambda}\bigl(\widehat{\psi}_{n}(\lambda)\bigr)^{-\delta},\quad \lambda>0.
\]
Let us now define 
\begin{equation}\label{theta-n}
\phi_{n+1}(t)=\int_0^1 \varphi_{n}^\delta(t)d\delta,\quad t>0.
\end{equation}
 Since the class of completely monotonic functions is closed under addition and pointwise limits, we conclude that $\phi_{n+1}\in(\mathcal{CM})$. In consequence, it follows from \cite[Theorem 5.4 and Theorem 5.5]{Gripenberg-Londen-Staffans-1990} that there exists $\psi_{n+1}\in(\mathcal{CM})$ such that 
 \begin{equation}\label{phin-psin}
 \psi_{n+1}\ast\varphi_{n+1}=1.
 \end{equation}
Furthermore, we have that 
\[
\widehat{\phi_{n+1}}(\lambda)=\int_0^1 \widehat{\varphi^{\delta}_{n}}(\lambda) d\delta=\int_0^1 \dfrac{1}{\lambda}\Bigl(\widehat{\psi}_{n}(\lambda)\Bigr)^{-\delta} d\delta=\dfrac{\widehat{\psi}_{n}(\lambda)-1}{\lambda\, \widehat{\psi}_{n}(\lambda)\,\log\bigl(\widehat{\psi}_{n}(\lambda)\bigr)},\quad \lambda>0.
\]
Since $(\phi_{n+1},\psi_{n+1})\in(\mathcal{PC})$, this in turn implies that 
\begin{equation}\label{Psi-n}
\widehat{\psi_{n+1}}(\lambda)=\dfrac{\widehat{\psi}_{n}(\lambda)\log\bigl(\widehat{\psi}_{n}(\lambda)\bigr)}{\widehat{\psi}_{n}(\lambda)-1},\quad \lambda>0.
\end{equation}
Therefore, we have that 
\[
\widehat{\psi_{n+1}}(t^{-1})=\dfrac{\widehat{\psi_{n}}(t^{-1})\log\bigl(\widehat{\psi_{n}}(t^{-1})\bigr)}{\widehat{\psi_{n}}(t^{-1})-1},\quad t>0.
\]
We note that the inductive hypothesis implies that 
\[
\dfrac{\widehat{\psi}_{n}(t^{-1})}{\widehat{\psi}_{n}(t^{-1})-1}\sim \dfrac{\log^{[n]}(t)}{\log^{[n]}(t)-1}\sim 1,\ \text{as}\ t\to \infty. 
\]

In consequence we have
\[
\widehat{\psi_{n+1}}\bigl(t^{-1}\bigr)\sim\log\bigl(\widehat{\psi}_{n}(t^{-1})\bigr)\sim \log^{[n+1]}(t),\ \text{as}\ t\to\infty. 
\]
On the other hand, since $\widehat{\psi_n}(t^{-1})\to 0$ as $t\to0^+$, we have that
\begin{align*}
\widehat{\psi_{n+1}}\bigl(t^{-1}\bigr)&\sim\widehat{\psi}_{n}(t^{-1})\log\bigl(\widehat{\psi}_{n}(t^{-1})\bigr),\ \text{as}\quad t\to0^+,
\end{align*}
which by the inductive hypothesis is equivalent to 
\begin{align*}
\widehat{\psi_{n+1}}\bigl(t^{-1}\bigr)&\sim (t\, \bigl(\log(t^{-1})\bigr)^{n})\Bigl(\log(t)+\log(\log(t^{-1})^{n})\Bigr),\quad \text{as}\ t\to0^+.
\end{align*}
We recall that 
\[
\lim_{t\to 0^+}\dfrac{\log(\log(t^{-1})^{n})}{\log(t)}=0, 
\]
for all $n\in\mathbb{N}$. Hence  
\[
\widehat{\psi_{n+1}}(t^{-1})\sim t\, \bigl(\log(t^{-1})\bigr)^{n+1},\ \text{as}\ t\to0^+,
\]
 and the proof is complete.
\end{proof}
\begin{remark}\label{Lemma-tech} Let $n\in\mathbb{N}$ and $(\phi_n,\psi_n)\in(\mathcal{PC})$ given in Corollary \ref{Coro-phi-psi-n}. Since $\widehat{\psi_n}(t^{-1})\sim \log^{[n]}(t)$ as $t\to \infty$, we have that for all $n\in\mathbb{N}$ the functions $\psi_n\notin L_{1}(\mathbb{R}_+)$. 
\end{remark}


\begin{example}\label{Example-7}
Let $n\in\{2,3,\cdots\}$. Consider pair $(k,\ell)=(\phi_n,\psi_n)$ given in Corollary \ref{Coro-phi-psi-n}, then  
\[
\textrm{Var}[X(t)]\sim \dfrac{2\nu}{\eta}\log^{[n]}(t), \quad \text{as}\quad t\to \infty,
\]
and 
\[
\textrm{Var}[X(t)]\sim \nu t^2\, \bigl(\log(t^{-1})\bigr)^{2n}, \quad \text{as}\quad t\to 0^+,
\]
where $\log^{[n]}=\underbrace{\log\circ\log\circ\dots\circ\log}_{n-\rm{times}}$.
\end{example}

\begin{proof}We note that for every $n\ge 2$ the function $\widehat{\psi_n}$ is a regularly varying function of index $\varrho=-1$. Indeed, recall that $\widehat{\psi}_1$ is defined by 
\[
\widehat{\psi}_1(t)=\dfrac{\log(t)}{t-1},\quad t>0.
\]
Hence, it clear that $\widehat{\psi}_1$ is a regularly varying function of index $\varrho=-1$. Assume now that $\widehat{\psi}_n$ is a regularly varying function of index $\varrho=-1$. Since $\widehat{\psi_n}(t)\to 0$ as $t\to\infty$, we have that 
\[
t\mapsto\dfrac{\widehat{\psi}_n(t)}{\widehat{\psi}_n(t)-1},\ t>0,
\]
is a regularly varying function of index $\varrho=-1$. Moreover, by properties of the logarithmic function we have that 
\[
t\mapsto \log(\widehat{\psi}_n(t)),\ t>0,
\]
is a slowly varying function. Therefore, it follows from \eqref{Psi-n} that $\widehat{\psi}_{n+1}$ is a regularly varying function of index $\varrho=-1$. On the other hand, it follows from Corollary \ref{Coro-phi-psi-n} that
\begin{equation}\label{psi:n:00}
\widehat{\psi}_n(t^{-1})\sim \log^{[n]}(t),\ \text{as} \ t\to\infty.
\end{equation}
Therefore, $\widehat{\ell}(t^{-1})$ is a slowly varying function and Theorem \ref{Theo-Asymptotic-Behavior} implies that 
\[
\textrm{Var}[X(t)]\sim \dfrac{2\nu}{\eta}\log^{[n]}(t), \quad \text{as}\quad t\to \infty,
\]
and 
\[
\textrm{Var}[X(t)]\sim \nu t^2\, \bigl(\log(t^{-1})\bigr)^{2n}, \quad \text{as}\quad t\to 0^+.
\]
\end{proof}

\begin{corollary}\label{Coro-phi-psi-n-delta}For all $n\in\mathbb{N}$ and $\delta\in(0,1)$ there exists a pair $(\phi_{n}^\delta,\psi_{n}^\delta)\in(\mathcal{PC})$ such that $\phi^\delta_{n}\in(\mathcal{CM})$ and 
\[
\widehat{\psi^\delta_{n}}(t^{-1})\sim\bigl(\log^{[n]}(t)\bigr)^\delta,\ \text{as}\ t\to\infty,
\]
and
\[
\widehat{\psi_{n}}(t^{-1})\sim t^\delta\, \bigl(\log(t^{-1})\bigr)^{\delta n}\ \text{as}\ t\to0^+.
\]

\end{corollary}

\begin{proof}Let $n\in\mathbb{N}$ and $\delta\in(0,1)$. Consider the pair $(\phi_n,\psi_{n})\in (\mathcal{PC})$ given in Corollary \ref{Coro-phi-psi-n}. According to Lemma \ref{Lemma-Comp-f-g} there are completely monotonic functions $h_{1n}^{\delta}$ and $h_{2n}^{\delta}$ such that 
\[
\widehat{h_{1n}^{\delta}}(\lambda)=\bigl(\widehat{\phi_{n}}(\lambda)\bigr)^{\delta},\quad \lambda>0,
\]
and
\[
\widehat{h_{2n}^{\delta}}(\lambda)=\bigl(\widehat{\psi_{n}}(\lambda)\bigr)^\delta, \quad \lambda>0.
\]
Now define the kernels
\begin{equation}\label{theta:log:log:delta}
\phi_n^\delta=g_{1-\delta}\ast h_{2n}^{\delta},\quad \text{and}\quad  \psi_n^\delta=h_{1n}^{\delta}.
\end{equation}
Applying directly the Laplace transform, we have 
\begin{equation*}
\widehat{\phi_n^\delta}(\lambda)=\dfrac{1}{\lambda}\left(\widehat{\phi_n}(\lambda)\right)^\delta\quad\text{and }\quad \widehat{\psi_n^\delta}(\lambda)=\left(\widehat{\psi_n}(\lambda)\right)^\delta ,\quad \lambda>0.
\end{equation*}

We note that by construction $\psi_n^\delta\in(\mathcal{CM})$. Hence, it follows from \cite[Theorem 5.4 and Theorem 5.5]{Gripenberg-Londen-Staffans-1990} that $\phi^\delta_n\in(\mathcal{CM})$. In consequence $(\phi^\delta_n,\psi^\delta_n)\in(\mathcal{PC})$. The rest of the proof follows the same ideas of Corollary \ref{Coro-phi-psi-n}.
\end{proof}

\begin{example}\label{Example-8}
Let $n\in\{2,3,\cdots\}$ and $\delta\in(0,1)$. Consider pair $(k,\ell)=(\phi_n,\psi_n)$ given in Corollary \ref{Coro-phi-psi-n-delta}, then  
\[
\textrm{Var}[X(t)]\sim \dfrac{2\nu}{\eta}\bigl(\log^{[n]}(t)\bigr)^\delta, \quad \text{as}\quad t\to \infty,
\]
and 
\[
\textrm{Var}[X(t)]\sim 2\nu \dfrac{t^{2\delta}\, \bigl(\log(t^{-1})\bigr)^{2n\delta}}{\Gamma(1+2\delta)}, \quad \text{as}\quad t\to 0^+,
\]
where $\log^{[n]}=\underbrace{\log\circ\log\circ\dots\circ\log}_{n-\rm{times}}$.
\end{example}

\begin{proof} Following the same ideas of Corollary \ref{Coro-phi-psi-n} we note that for all $n\in\mathbb{N}$ the functions $\widehat{\psi_n^\delta}$ are regularly varying functions of index $\varrho=-\delta$ and $\widehat{\psi_n^\delta}(t^{-1})$ is a slowly varying function. The rest of the proof is similar to the proof of Example \ref{Example-7}.
\end{proof}


\section{Application to ultra slow diffusion equations}

There are another contexts where the pairs $(k,\ell)\in(\mathcal{PC})$ play a fundamental role. For example, this type of functions has been successfully exploited to study the so-called {\it sub-diffusion processes}, see \cite{Kemppainen-Siljander-Vergara-Zacher-2016,Pozo-Vergara-2019,Vergara-Zacher-2015} and references therein. In order to fix some ideas and explain why the results developed in this work could be interesting in the theory of subdiffusion processes, we consider the following equation
\begin{align}
\partial_t (k\ast(u(\cdot,x)-u_0(x)))(t)-\Delta u(t,x)&=0,\quad t>0, x\in\mathbb{R}^d,\label{Sub-Diffusion}\\
u(0,x)&=u_0(x), \quad x\in\mathbb{R}^d,
\end{align}
where $k$ is a kernel of type $(\mathcal{PC})$ and $u_0$ is a given function. It has been proved in \cite[Section 2]{Kemppainen-Siljander-Vergara-Zacher-2016} that the fundamental solution of \eqref{Sub-Diffusion} coincides with the probability density function of a stochastic process $X(t)$. Moreover, in \cite[Lemma 2.1]{Kemppainen-Siljander-Vergara-Zacher-2016} it has been proved that the {\it mean square displacement} $M(t)$ of such process is given by 
\begin{equation}\label{MSD}
M(t)=2d(1\ast \ell)(t),\quad t\ge 0.
\end{equation}
The function $M(t)$ allows to measure how fast or slow is the diffusion of the equation \eqref{Sub-Diffusion}. It is worthwhile to mention that the slowest known rate of growth of $M(t)$ follows a logarithmic law, for instance see \cite{Kochubei-2008} and references therein. In such work Kochubei considered functions of the form 
\[
k(t)=\int_0^1 g_{\alpha}(t)\sigma(\alpha)d\alpha,\quad t>0,
\]
where $\sigma(t)$ with $t\in[0,1]$ is a continuous, non-negative function different from zero on a set of positive measure.

Those equations of the form \eqref{Sub-Diffusion} whose mean square displacement follows a logarithmic rate (or even slower) are known in the specialized literature as {\it ultra slow diffusion equations} and they are strongly related with {\it ultraslow inverse subordinators}, see \cite{Meerschaert-Scheffler-2006}.  

Our work allows to study some ultra-slow diffusion equations which are not considered before. For instance, the equation \eqref{Sub-Diffusion} with $(k,\ell)=(\phi_n,\psi_n)$ for some $n\in\mathbb{N}$, where $(\phi_n,\psi_n)$ has been defined in Corollary \ref{Coro-phi-psi-n}. According to \eqref{MSD} we have that the Laplace transform of $M$ is given by
\[
\widehat{M}(\lambda)=\frac{2d}{\lambda}\,\widehat{\psi_n}(\lambda),\quad\lambda>0,
\]
which by Karamata-Feller's Theorem \ref{Karamata:Feller} and the asymptotic behavior of $\widehat{\psi_n}$ given in \eqref{psi:n:00} imply that 
\[
M(t)\sim 2d\, \log^{[n]}(t),\quad \text{as}\quad t\to \infty.
\]
In consequence, the mean square displacement $M(t)$ grows (at infinity) slower than a logarithmic function. 

This procedure can be applied to all the pairs of functions in $(\mathcal{PC})$ defined in the Corollary \ref{Coro-phi-delta} and Corollary \ref{Coro-phi-psi-n-delta}. As far we know, this implies that there are an infinite number of ultra-slow diffusion equations which have not been analyzed before. All these new interesting examples will be  studied in a forthcoming work.

\bibliographystyle{amsplain}
\bibliography{Ref-Francisco}


\end{document}